\documentclass[reqno,12pt]{amsart}
\usepackage{amsmath, amssymb, tikz}
\usepackage{amsthm, amsbsy}
\usepackage{graphicx}
\usepackage{float}
\usepackage{stackrel}

\theoremstyle{plain}
\newtheorem{theorem}{\indent\bf Theorem}[section]
\newtheorem{lemma}[theorem]{\indent\bf Lemma}
\newtheorem{corollary}[theorem]{\indent\bf Corollary}

\theoremstyle{definition}

\newtheorem{remark}[theorem]{\indent\bf Remark}

\begin{document}
\title[Zero-free regions and irreducibility criteria]{Zero-free angular sectors and lens-shaped regions for polynomials, with applications to irreducibility
}
\author[C.M. Bonciocat]{Ciprian Mircea Bonciocat}
\address{Address: University of California, Los Angeles, CA 90095}
\email{Ciprian\_Bonciocat@yahoo.com}

\author[N.C. Bonciocat]{Nicolae Ciprian Bonciocat}
\address{Simion Stoilow Institute of Mathematics of the Romanian Academy, Research Unit 7, P.O. Box 1-764, Bucharest 014700, Romania}
\email{Nicolae.Bonciocat@imar.ro}

\dedicatory{Dedicated to the memory of Professor Doru \c Stef\u anescu}

\keywords{irreducible polynomials; prime numbers; zero-free sectors}
\subjclass[2010]{Primary 11R09; Secondary 11C08.}
\begin{abstract}
For a real polynomial $f$ we present explicit zero-free angular sectors in the complex plane, symmetric with respect to the real axis, with angles depending only on the degree of $f$, and vertices expressed in terms of the coefficients of $f$.
We also describe zero-free lens-shaped regions for $f$ that are associated to the zero-free sectors of the reciprocal of $f$. As an application, we use these zero-free regions to obtain irreducibility criteria for polynomials with integer coefficients that take a prime or a prime power value.
\end{abstract}
\maketitle
\section{Introduction}

Some of the classical or more recent irreducibility criteria for polynomials $f$ with integer coefficients rely on the existence of a suitable prime or prime power in the factorization of $f(m)$ for some integer $m$, and on suitable estimates on the distance between $m$ and the roots of $f$.  The first such criteria appeared in the works of St\"ackel \cite{Stackel},  P\'{o}lya and Szeg\"{o} \cite{PolyaSzego}, Ore \cite{Ore}, Weisner \cite{Weisner} and Dorwart \cite{Dorwart}. Some of the most famous and appealing criteria of this type come from writing certain families of integers in number systems with base $B$, and then replacing the base by an indeterminate. For instance, in 
\cite{PolyaSzego} P\'{o}lya and Szeg\"{o} present the following elegant irreducibility criterion of A. Cohn:\medskip

{\bf Theorem 1.} \ \emph{If we express a prime $p$ in the decimal system as 
\[
p=\sum\limits_{i=0}^{n}a_{i}10^{i},\quad0\leq a_{i}\leq9,
\]
then the polynomial $\sum_{i=0}^{n}a_{i}X^{i}$ is irreducible in $\mathbb{Z}[X]$.}
\medskip

This result was extended to an arbitrary base $B$ by Brillhart, Filaseta and Odlyzko \cite{Brillhart}: \medskip

{\bf Theorem 2.} \ \emph{If we express a prime $p$ in the number system with base $B\geq2$ as 
\[
p=\sum\limits_{i=0}^{n}a_{i}B^{i},\quad0\leq a_{i}\leq B-1,
\]
then the polynomial $\sum_{i=0}^{n}a_{i}X^{i}$ is irreducible in $\mathbb{Z}[X]$.}
\medskip

Another generalization of Cohn's theorem was achieved in \cite{Filaseta1} by replacing the prime $p$ with a composite number $pq$ with $q<B$.
Further generalizations \cite{Brillhart}, \cite{Filaseta2}, allow the coefficients of $f$ to be different from digits. For instance, Filaseta \cite{Filaseta2} obtained the following result for polynomials with non-negative coefficients.\medskip

{\bf Theorem 3.} \ \emph{Let $f(X)=\sum_{i=0}^{n}a_{i}X^{i}$ be such that 
$f(10)$ is a prime. If the $a_{i}$'s satisfy $0\leq a_{i}\leq a_{n}10^{30}$
for each $i=0,1,\dots,n-1$, then $f(X)$ is irreducible.} \medskip

Cole, Dunn, and Filaseta \cite{CDF} obtained sharp bounds $M(b)$ depending on an integer $b\in [3,20]$ such that if $f(b)$ is prime and each coefficient of $f$ is non-negative and at most $M(b)$, then $f$ is irreducible. Several irreducibility criteria for polynomials 
that take a prime or a prime power value and have a coefficient of sufficiently large absolute value have been obtained in \cite{Bonciocat1} 
and \cite{Bonciocat2}. For instance, the following result was proved in \cite{Bonciocat1}:
\medskip

{\bf Theorem 4.} \ \emph{If we write a prime number as a sum of integers $
a_{0},\dots,a_{n}$, with $a_{0}a_{n}\neq 0$ and $|a_{0}|>
\sum_{i=1}^{n}|a_{i}|2^{i}$, then the polynomial $\sum_{i=0}^{n}a_{i}X^{i}$
is irreducible over $\mathbb{Q}$.} 
\medskip

A unifying approach using the concept of admissible triples to study irreducibility of polynomials in terms of the prime factorization of the value that they take at a certain integer argument was recently developed by Guersenzvaig \cite{Guersenzvaig}. In \cite{Guersenzvaig} one may also find upper bounds for the total number of irreducible factors (counted with multiplicities) for some classes of integer polynomials (for the study of roots multiplicities and square free factorization we refer the reader to Guersenzvaig and  Szechtman \cite{GS}, and to Mignotte and \c Stef\u anescu \cite{MS}, for instance). Further connections between prime numbers and irreducible polynomials, some of them related to Schinzel's hypothesis, may be found in the work of Murty \cite{RamMurty}, Girstmair \cite{Girstmair}, and Bodin, D\`ebes and Najib \cite{BDN3}, for instance. 

A key factor in proving irreducibility criteria of this type is to obtain sharp estimates for the location of the roots of $f$. Among the earliest estimates for the maximum of the absolute values of the roots of a given polynomial we mention here the bounds due to Cauchy and Lagrange. An improvement of the bound of Lagrange for the largest absolute value of the roots of a polynomial was recently obtained by Batra, Mignotte, and \c Stef\u anescu \cite{BMS}. For additional estimates, of which some rely on the use of some families of parameters, we will only mention here the classical methods of  Ballieu \cite{Ballieu}, \cite{Marden}, Fujiwara \cite{Fujiwara}, Cowling and Thron \cite{Cowling1}, \cite{Cowling2}, Kojima \cite{Kojima}, or methods using estimates for the characteristic roots for complex matrices \cite{Perron} applied to the companion matrix of a polynomial. There are many other useful results in the literature that provide bounds for the real parts or for the imaginary parts of the roots, and also for the positive roots of a polynomial, of which we will only mention the results of Tur\' an \cite{Turan}, L\' aszl\' o \cite{Laszlo}, Kioustelidis \cite{Kioustelidis} and \c Stef\u anescu \cite{DStef1}. Other classical related results refer to the number of zeros of a polynomial in a sector. For such results, as well as generalizations to complex polynomials of Sturm's Theorem, and of Descartes' Rule, we refer the reader to Marden \cite{Marden}, chapter IX.

The aim of this paper is twofold. Our first aim is to obtain explicit sectors in the complex plane (with vertices on the real axis) where a real polynomial $f$ of degree $n$ has no roots. In the general case, the $x$-coordinate $v_f$ of the vertex of such a sector will depend on the coefficients of $f$, while the angle of the sector will only depend on the degree of $f$. We will search for families of polynomials $f(X)=a_0+a_1X+\cdots +a_nX^n$ for which $v_f$ is independent on the coefficients $a_i$, provided that the coefficients satisfy certain inequalities. This is the case, for instance, of the polynomials whose coefficients are all non-negative. As a counterpart for a zero-free sector of a polynomial $f$ we describe a zero-free lens-shaped region for $f$, that is obtained by inversion ($z \mapsto \frac 1z$) with respect to the origin of the complex plane of a zero-free sector for the reciprocal $\tilde{f}$ of $f$. The results obtained in this respect will appear in Section \ref{MulteSectoare}.
Our second goal is to use this kind of information on the location of the roots of $f$ to derive irreducibility conditions for $f$ in the case that the coefficients $a_i$ are integers, and the prime factorization of $f(m)$ is known for some suitable integer argument $m$ inside the zero-free sector or inside the zero-free lens-shaped region associated to $f$. To get a glimpse of the results that we obtain, we will only mention here the simplest three irreducibility criteria that are corollaries of more general results obtained in Section \ref{trei}.
\smallskip

{\em Let $f(X)$ be a polynomial with non-negative integer coefficients, of degree $n\geq 2$. If $f(m)$ is a prime number for some integer $m>\frac{1}{\sin\frac{\pi}{n}}$, then $f$ is irreducible over $\mathbb{Q}$. 
}
\smallskip

{\em Let $f(X)=a_0+a_1X+\cdots +a_nX^n\in \mathbb{Z}[X]$ with $a_0\neq 0$, and denote by $j_1 < j_2 < \cdots < j_\ell$ all the indices $j$ for which $a_j < 0$. If $a_n>|a_{j_1}+\cdots + a_{j_\ell}|$ and $f(m)$ is a prime number for an integer $m > 1+\frac{1}{\sin{\frac{\pi}{n}}}$,
then $f$ is irreducible over $\mathbb{Q}$.}
\smallskip

{\em Let $f(X)=a_0+a_1X+\cdots +a_nX^n\in \mathbb{Z}[X]$ be a polynomial of degree $n\geq 2$, with all partial sums $a_n+a_{n-1}+\cdots +a_{n-j}$ non-negative, where $0\leq j\leq n$. If for some integer $m>1+\frac{1}{\sin \frac{\pi}{n}}$ we have $f(m)=p^k$ with $p$ a prime number, $k$ a positive integer and $p\nmid f'(m)$, then $f$ is irreducible over $\mathbb{Q}$. }

\smallskip

A series of examples of infinite families of irreducible polynomials obtained from the irreducibility criteria in Section \ref{trei} will be provided in the last section of the paper. We mention here that we will only restrict our attention to finding zero-free sectors and lens-shaped regions that are situated in the right half-plane. Correspondingly, we will only study positive integers $m$ for which the prime factorization of $f(m)$ is suitable to derive irreducibility conditions for $f$. One may obtain similar results that allow the integer $m$ to be negative, and the zero-free sectors and lens-shaped regions to be situated in the left half-plane, by considering instead of $f(X)$ the polynomial $f(-X)$.

\section{Zero-free sectors and lens-shaped regions for polynomials}\label{MulteSectoare}
	
This section is devoted to finding explicit sectors $S_{v_{f},\theta }$ in the complex plane with vertex $v_{f}$ on the real axis, symmetric with respect to the real axis, and having total angle $2\theta $, where an $n$-degree polynomial $f(X)$ with real coefficients cannot vanish.

We will use the following notation. 

For two real numbers $v,\theta $ with $0<\theta <\pi$ we define the following three sectors
\begin{eqnarray*}
S_{v,\theta } & = & \left\{ z=v+\rho e^{i\varphi }: \rho >0, \ |\varphi |<\theta \right\} \\
S_{v,\theta }^+ & = & \left\{ z=v+\rho e^{i\varphi }: \rho >0, \ 0<\varphi <\theta \right\} ,\\
S_{v,\theta }^- & = & \left\{ z=v+\rho e^{i\varphi }: \rho >0, \ 0>\varphi >-\theta \right\} .
\end{eqnarray*}\vspace{-1.5em}
\begin{figure}[h]
\begin{tikzpicture}[>=stealth]
\draw[->] (-0.5, -1) -- (-0.5, 1.5); \draw[->] (-1, 0) -- (3.5, 0);
\begin{scope}[shift={(-0.5,0)}]
	\draw[fill] (0.7, 0) circle (0.05);
	\draw[thick, fill=gray, fill opacity = 0.2] (3, 0.8) -- (0.7, 0) -- (3, -0.8);
	\draw[thick, dashed] (3, 0.8) -- (3.5, 0.96)    (3, -0.8) -- (3.5, -0.96);
	\draw (0.7, 0) node[anchor=45] {$v$};
	\draw (1.7, 0.04) node[anchor=217] {$\theta$};
	\draw (1.8, -0.9) node {$S_{v, \theta}$};
\end{scope}
\end{tikzpicture}\;\;
\begin{tikzpicture}[>=stealth]
\draw[->] (-0.5, -1) -- (-0.5, 1.5); \draw[->] (-1, 0) -- (3.5, 0);
\begin{scope}[shift={(-0.5,0)}]
	\draw[fill] (0.7, 0) circle (0.05);
	\draw[thick, fill=gray, fill opacity = 0.2] (3, 0.8) -- (0.7, 0) -- (3, 0);
	\draw[thick, dashed] (3, 0.8) -- (3.5, 0.96)    (3, 0) -- (3.5, 0);
	\draw (0.7, 0) node[anchor=45] {$v$};
	\draw (1.7, 0.04) node[anchor=217] {$\theta$};
	\draw (1.8, -0.9) node {$S^+_{v, \theta}$};
\end{scope}
\end{tikzpicture}\;\;
\begin{tikzpicture}[>=stealth]
\draw[->] (-0.5, -1) -- (-0.5, 1.5); \draw[->] (-1, 0) -- (3.5, 0);
\begin{scope}[shift={(-0.5,0)}]
	\draw[fill] (0.7, 0) circle (0.05);
	\draw[thick, fill=gray, fill opacity = 0.2] (3, -0.8) -- (0.7, 0) -- (3, 0);
	\draw[thick, dashed] (3, -0.8) -- (3.5, -0.96)    (3, 0) -- (3.5, 0);
	\draw (0.7, 0) node[anchor=45] {$v$};
	\draw (1.7, -0.035) node[anchor=-217] {$\theta$};
	\draw (1.8, -0.9) node {$S^-_{v, \theta}$};
\end{scope}
\end{tikzpicture}\;\;
\caption{The sectors $S_{v, \theta}$, $S^+_{v, \theta}$ and $S^-_{v, \theta}$}\label{fig1}
\end{figure}
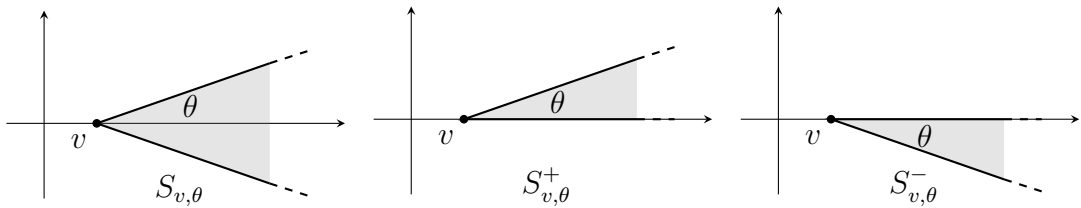

With this notation we have the following two well-known lemmas, the first one being a special case of a more general result of Kempner \cite{Kempner} on complex polynomials that do not vanish on a certain ray $\arg z=\omega $. Even if they are quite elementary, for the sake of completeness we will give proofs for both of them.
\begin{lemma}\label{lemaReal(f)}
If $f(X)\in \mathbb{R}[X]$ is of degree $n\geq 1$ and has non-negative coefficients, then $\Re (f(z))>0$ for any $z$ in the sector $S_{0,\frac{\pi }{2n}}$. In particular, $f$ has no zero in $S_{0,\frac{\pi }{2n}}$.
\end{lemma}
\begin{proof}
Assume that $f(X)=a_0+a_1X +\cdots +a_nX^n$ with $a_n>0$. For a complex number $z\in S_{0,\frac{\pi }{2n}}$, all of  its powers $z^j$ with $j=0,1,\dots ,n$ lie in the right half-plane $\Re(z)>0$, and the same holds for all the complex numbers $a_jz ^j$ with $a_j\neq 0$, as $a_j$ were assumed to be non-negative. As a consequence, $\Re(a_0+a_1z +\cdots +a_nz^n)>0$, and in particular, $z$ cannot be a root of $f$.
\end{proof}
\begin{lemma}{\em (}\cite [Lemma 2] {Filaseta2}{\em )}\label{lemaImaginar(f)}
If $f(X)\in \mathbb{R}[X]$ is of degree $n\geq 1$ and has non-negative coefficients, then $\Im (f(z))>0$ for all $z$ in the sector $S_{0,\frac{\pi }{n}}^{+}$, and $\Im (f(z))<0$ for all $z$ in the sector $S_{0,\frac{\pi }{n}}^{-}$. In particular, $f$ has no zero in the sector $S_{0,\frac{\pi }{n}}$.
\end{lemma}
\begin{proof}
Assume again that $f(X)=a_0+a_1X +\cdots +a_nX^n$ with $a_n>0$. For a complex number $z\in S_{0,\frac{\pi }{n}}^+$, all the complex numbers $a_jz^j$ with $j\in \{1,\dots ,n\}$ and $a_j\neq 0$ lie in the upper half-plane $\Im(z)>0$, so for such $z$ we have $\Im(a_0+a_1z +\cdots +a_nz^n)>0$. Similarly, for $z\in S_{0,\frac{\pi }{n}}^-$, all the complex numbers $a_jz ^j$ with $j\in \{1,\dots ,n\}$ and $a_j\neq 0$ lie in the lower half-plane $\Im(z)<0$, so $\Im(a_0+a_1z +\cdots +a_nz^n)<0$. Finally, since $f$ has no positive real roots, we conclude that it has no zeros in $S_{0,\frac{\pi }{n}}$.
\end{proof}
Some remarks are in order. For a polynomial $f$ of degree $n\geq 1$ with real coefficients of arbitrary signs, it is in general way more difficult to find a sharp sector where it has no roots. The most natural way to search for such sectors is to use translates of the indeterminate that will lead to a new polynomial $f(X+\alpha )$ all of whose non-zero coefficients have the same sign. Thus, if $f(X)=a_0+a_1X+\cdots +a_nX^n$, then $f(X+\alpha )=b_0+b_1X+\cdots +b_nX^n$ with coefficients $b_k$ given by
\[
b_k=\sum\limits_{i=0}^{n-k}\alpha ^{i}a_{k+i}\tbinom{k+i}{i},\quad k=0,\dots ,n.
\]
Assume that $a_n>0$. For the coefficients $b_0,\dots ,b_k$ to be non-negative, we may for instance ask $\alpha $ to exceed all the real roots of the polynomials 
\[
g_k(X)=\sum\limits_{i=0}^{n-k}a_{k+i}\tbinom{k+i}{i}X^{i}, \quad k=0,\dots ,n.
\]
Once we have found such an $\alpha $, we may conclude that $f$ has no roots in the sector $S_{\alpha ,\frac{\pi}{n}}$, but finding an $\alpha $ as small as possible with this property is by no means an easy task, and here some methods of estimating the real roots of a real polynomial might be useful. Another idea is to choose $\alpha $ to be an upper bound for the real parts of all the roots of $f$. As $a_n>0$, and the non-real roots of $f$ come in conjugate pairs, it follows that $f(X+\alpha )$ has only non-negative coefficients. However, finding a reasonably sharp such $\alpha $ is in general quite difficult as well.  

One slightly simpler method to find a suitable $\alpha $ is given by the following lemma. 
\begin{lemma}\label{lemapozitivagenerala}
Let $f(X)=a_0+a_1X+\cdots +a_nX^n$ be a polynomial of degree $n\geq 1$ with real coefficients, and $\alpha $ a positive real number such that all the sums of the form
\[
\alpha ^{j}a_n+\alpha ^{j-1}a_{n-1}+\cdots +a_{n-j}
\]
are non-negative, where $0\leq j\leq n$. Then  $f(X+\alpha )$ has non-negative coefficients, so $f$ has no zeros in $S_{\alpha ,\frac{\pi }{n}}$. 
\end{lemma}
\begin{proof}
Let us denote the sum $\alpha ^{j}a_n+\alpha ^{j-1}a_{n-1}+\cdots +a_{n-j}$ by $s_{f,j,\alpha}$ for each $0\leq j\leq n$. 
We argue by induction on $n$: if $n=1$, the inequalities $a_1\geq 0$ and $\alpha\cdot a_1+a_0\geq 0$ are equivalent to the assertion that the coefficients of $f(X+\alpha )=a_1X+\alpha a_1+a_0$ are non-negative; so assume that $n>1$, and that the statement has already been proven for smaller degrees. Note that one may write
\[
f(X)=a_n(X^n-\alpha X^{n-1})+[(\alpha a_n+a_{n-1})X^{n-1}+a_{n-2}X^{n-2}+\cdots +a_1X+a_0].
\]
Denote the polynomial in square brackets by $g(X)$, and observe that $s_{g,j,\alpha}=s_{f,j+1,\alpha}$ for all $j\in\{ 0,\dots ,n-1\} $. Thus, $g$ satisfies the inductive hypothesis, and so $g(X+\alpha )$ has non-negative coefficients. It thus suffices to prove the statement for the remaining term $a_n(X^n-\alpha X^{n-1})$, where $a_n\geq 0$ by hypothesis. This is straight-forward, as
\[
f(X+\alpha )-g(X+\alpha )=a_nX(X+\alpha )^{n-1}
\]
has positive coefficients, so the induction is complete. In particular, by Lemma \ref{lemaImaginar(f)}, $f(X+\alpha )$ has no zeros in the sector $S_{0,\frac{\pi}{n}}$, so $f$ has no zeros in the sector $S_{\alpha,\frac{\pi}{n}}$, as claimed. 
\end{proof}

We note here that a positive real number $\alpha $ as in the statement of Lemma \ref{lemapozitivagenerala} always exists. For instance, one may choose $\alpha$ to exceed all the real roots (or the absolute values of all the roots) of the polynomials $f_j(X)=a_nX^{j}+a_{n-1}X^{j-1}+\cdots + a_{n-j}$, $j=0,\dots ,n$. Here too, finding an $\alpha $ as small as possible with the properties required in Lemma \ref{lemapozitivagenerala} is in general a difficult task. However, for $\alpha =1$, the conditions on the coefficients of $f$ in Lemma \ref{lemapozitivagenerala} take a simpler form, and this will be used in the following section. 

The remaining part of this section is devoted to obtaining several explicit sectors where a polynomial has no roots, provided some additional information on the signs of its coefficients is known. The method that we will employ relies on a suitable choice of some families of parameters, and is inspired by the classical results of Fujiwara \cite{Fujiwara} that establish bounds for the absolute values of the roots of a polynomial. To do this, we will need one more elementary lemma, concerning polynomials of the form $f(X)=X^n-\alpha X^{\ell}$ with $\alpha \geq 0$.

\begin{lemma}\label{douamonoameGenerala}
Let $n > \ell $ be non-negative integers, and $\alpha$ a non-negative real number. Then 

\noindent i)\ \ $\Re(z^n - \alpha z^\ell ) > 0$ for all complex numbers $z\in S_{\alpha^{\frac{1}{n-\ell}},\frac{\pi }{2n}}$;

\noindent ii) $\Im(z^n - \alpha z^\ell ) > 0$ for all $z\in S_{\alpha^{\frac{1}{n-\ell}},\frac{\pi }{n}}^+$, and $\Im(z^n - \alpha z^\ell ) < 0$ for all $z\in S_{\alpha^{\frac{1}{n-\ell}},\frac{\pi }{n}}^-$.
\end{lemma}	
\begin{proof}
The case $\alpha =0$ follows directly from Lemma \ref{lemaReal(f)} and Lemma \ref{lemaImaginar(f)}, so let us assume that $\alpha >0$. Let $f(X)=X^n-\alpha X^{\ell }$, and let us note that 
\[
f(X+\alpha ^{\frac{1}{n-\ell }})=(X+\alpha ^{\frac{1}{n-\ell }})^{\ell }\sum\limits _{k=0}^{n-\ell -1}\tbinom{n-\ell }{k}X^{n-\ell -k}\alpha ^{\frac{k}{n-\ell }},
\]
which has only positive coefficients. The proof follows now by Lemma \ref{lemaReal(f)} and Lemma \ref{lemaImaginar(f)} applied to $f(X+\alpha ^{\frac{1}{n-\ell }})$.
\end{proof}
Our first result that provides zero-free sectors for a polynomial, provided we know the signs of its coefficients, is the following:
\begin{theorem} \label{SectorGenerala1}
Let $f(X) = a_n X^n + \cdots + a_1 X + a_0$ be a polynomial with real coefficients, with $a_n > 0$, and denote by $j_1 < j_2 < \cdots < j_\ell$ the indices $j$ for which $a_j < 0$. Let
	\begin{equation}\label{alaFujiwara}
		v= \max \limits _{1\leq i\leq\ell}\left(\frac{|a_{j_1} + a_{j_2} + \cdots + a_{j_\ell}|}{a_n}\right)^{\frac{1}{n-j_{i}}}.
	\end{equation}
Then $f$ has no zero in the sector $S_{v,\frac{\pi}{n}}$. 
\end{theorem}
\begin{proof} If all the coefficients of $f$ are non-negative, the result follows immediately from Lemma \ref{lemaImaginar(f)}, so we may assume that there is at least one negative coefficient. The strategy is to break the polynomial $f$ into a sum of polynomials of the same form as in Lemma \ref{douamonoameGenerala}. To this end, let $\lambda_0, \ldots, \lambda_{n-1}$ be non-negative real numbers to be chosen later, which satisfy $\lambda_0 + \cdots + \lambda_{n-1} = 1$, and $\lambda_k = 0$ iff $a_{k} \ge 0$. Then we may write $f$ as
	\begin{eqnarray*}
		f(X) & = & \sum_{k = 0}^{n-1} (\lambda_k a_{n} X^n + a_{k} X^{k}) \\
 & = & \sum\limits_{a_{k} < 0} \lambda_k a_{n} \left( X^n - \frac{|a_{k}|}{\lambda_k a_{n}} X^{k} \right) + \sum\limits_{a_{k} \ge 0,\ k<n} a_{k} X^{k}.
	\end{eqnarray*}
By Lemma \ref{douamonoameGenerala} i) we first see that $z^n - \frac{|a_{k}|}{\lambda_k a_{n}} z^{k}$ has positive real part in the sector $S_{v_k,\frac{\pi}{2n}}$ with 
\begin{equation}\label{vk}
v_k=\left(\frac{|a_{k}|}{\lambda_k a_{n}}\right)^{1/(n-k)},
\end{equation}
for all $k$ with $a_{k} < 0$, and similarly that $a_{k} z^{k}$ has positive real part in $S_{0,\frac{\pi}{2n}}$ for all $k<n$ with $a_{k}>0$, if such $k$ exist. As a consequence, $\Re (f(z))>0$ for any complex number $z\in S_{w,\frac{\pi}{2n}}$ with
\begin{equation}\label{w}
w=\max\limits _{k:a_k<0} \left(\frac{|a_{k}|}{\lambda_k a_{n}}\right)^{1/(n-k)}.
\end{equation}
In particular, $f$ has no real roots exceeding $w$, so it has no real roots in the sector $S_{w,\frac{\pi}{n}}$. It remains to prove that $f$ cannot have non-real roots in the sector $S_{w,\frac{\pi}{n}}$.
By Lemma \ref{douamonoameGenerala} ii) we then see that for any $k$ with $a_k<0$, the polynomial $z^n - \frac{|a_{k}|}{\lambda_k a_{n}} z^{k}$ has positive imaginary part for any $z$ in the sector $S_{v_k,\frac{\pi}{n}}^+$, and has negative imaginary part for any $z$ in the sector $S_{v_k,\frac{\pi}{n}}^-$, with $v_k$ given by (\ref{vk}). Besides, for all $k<n$ with $a_k>0$ (if any), $a_{k} z^{k}$ has positive imaginary part for any $z$ in the sector $S_{0,\frac{\pi}{n}}^+$, and negative imaginary part in the sector $S_{0,\frac{\pi}{n}}^-$. Observe now that $S_{w,\frac{\pi }{n}}^+\subseteq S_{v_k,\frac{\pi }{n}}^+$ and $S_{w,\frac{\pi }{n}}^-\subseteq S_{v_k,\frac{\pi }{n}}^-$ for each $k$ with $a_k<0$. Then, since we also have $S_{w,\frac{\pi }{n}}^+\subseteq S_{0,\frac{\pi }{n}}^+$ and $S_{w,\frac{\pi }{n}}^-\subseteq S_{0,\frac{\pi }{n}}^-$, we deduce that $f(z)$ has positive imaginary part on $S_{w,\frac{\pi }{n}}^+$, and negative imaginary part on $S_{w,\frac{\pi }{n}}^-$. Recalling that it has no real roots exceeding $w$, we conclude that $f$ has no roots in the sector $S_{w,\frac{\pi }{n}}$.

 We now find the right $\lambda_k$: one idea is to choose them so that the quotients $\frac{|a_{k}|}{\lambda_k a_{n}}$ do not depend on $k$, i.e. so that $\lambda_k$ are proportional to the $|a_{k}|$ whenever $a_{k} < 0$. Since at least one $a_{k}$ is negative, we can define
	\begin{equation*}
		\lambda_k = \frac{|a_{k}|}{\sum\limits _{a_{k}<0}|a_{k}|}\quad {\rm if}\ a_{k}<0,
	\end{equation*}
and $\lambda _k = 0$ otherwise. These $\lambda _k$ indeed sum up to one, and by plugging their expressions into (\ref{w}) we obtain in this case the value
	\begin{equation*}
		w=\max\limits _{1\leq i\leq \ell}\left(\frac{|a_{j_1}+a_{j_2}+\cdots +a_{j_\ell}|}{a_n}\right)^{\frac{1}{n-j_i}},
	\end{equation*}
which is precisely the expression of $v$ in the statement of the theorem.
\end{proof}
\begin{remark}\label{MaximLaCapete}
Let us denote $|a_{j_1}+\cdots +a_{j_{\ell}}|$ by $L_{-}(f)$. We notice that the maximum in (\ref{alaFujiwara}) is attained for $i=1$ if $a_n\geq L_{-}(f)$, and for $i=\ell$ if $a_n<L_{-}(f)$, hence it may be simply written as
\[
\max\left\{ \left(\frac{L_{-}(f)}{a_n}\right)^{\frac{1}{n-j_1}}, \left(\frac{L_{-}(f)}{a_n}\right)^{\frac{1}{n-j_\ell}}\right\} .
\]
\end{remark}
As one can see in the proof of Theorem \ref{SectorGenerala1}, we actually proved the following more general result, that depends on a suitable set of parameters.
\begin{theorem} \label{SectorGenerala2}
Let $f(X) = a_n X^n + \cdots + a_1 X + a_0$ be a polynomial with real coefficients, with $a_n > 0$, and denote by $j_1 < j_2 < \cdots < j_\ell$ the indices $j$ for which $a_j < 0$. Let also $\lambda _1,\dots ,\lambda _\ell $ be arbitrary positive real numbers with $\lambda _1+\cdots +\lambda _\ell =1$, and define
	\begin{equation}\label{alaFujiwara2}
		v=\max \limits _{1\leq i\leq\ell}\left(\frac{|a_{j_i}|}{\lambda _ia_n}\right)^{\frac{1}{n-j_{i}}}.
	\end{equation}
Then $f$ has no zero in the  sector $S_{v,\frac{\pi}{n}}$.
\end{theorem}
In an attempt to minimize the maximum in (\ref{alaFujiwara2}), one may choose different sets of parameters $\lambda _1,\dots ,\lambda _\ell $. For instance one may take $\lambda _i=\frac{1}{\ell}$ for each $i$, or $\lambda _i=\frac{1}{2^{\ell }}\cdot {\ell \choose i}$ for $i=1,\dots ,\ell $, or some other $\ell$-tuples of parameters that take into account the specific values of the negative coefficients $a_{j_1}, \dots ,a_{j_\ell}$. 

A result stronger than Theorem \ref{SectorGenerala1} can be obtained, if instead of $a_n$ we consider successively all the positive coefficients $a_j$ with $j>j_{\ell}$, and then take the minimum of the vertices of the sectors that we obtain. This situation is illustrated in the following result.
\begin{theorem} \label{ceamaigenerala}
Let $f(X) = a_n X^n + \cdots + a_1 X + a_0$ be a polynomial with real coefficients, with $a_n > 0$, and denote by $j_1 < j_2 < \cdots < j_\ell$ the indices $j$ for which $a_j < 0$, and by $k_1<k_2<\cdots <k_t=n$ the indices $k>j_{\ell }$ for which $a_k>0$. Let
	\begin{equation}\label{SectoralaFujiwaraGenerala}
		v=\min_{1\leq r\leq t} \ 
\max \limits _{1\leq i\leq\ell}\left(\frac{|a_{j_1} + a_{j_2} + \cdots + a_{j_\ell}|}{a_{k_r}}\right)^{\frac{1}{k_r-j_{i}}}.
	\end{equation}
Then $f$ has no zero in the sector $S_{v,\frac{\pi}{n}}$. 
\end{theorem}
\begin{proof}
Here we may assume again that $f$ has at least one negative coefficient, so for an index $k_r$ we may write $f$ as
	\begin{eqnarray*}
		f(X) & = & \sum_{i = 1}^{\ell} (\lambda_i a_{k_r} X^{k_r} + a_{j_i} X^{j_i}) + \sum\limits_{a_{k} \ge 0,\ k\neq k_r} a_{k} X^{k}\\
 & = & \sum\limits_{i=1}^\ell \lambda_i a_{k_r} \left( X^{k_r} - \frac{|a_{j_i}|}{\lambda_i a_{k_r}} X^{j_i} \right) + \sum\limits_{a_{k} \ge 0,\ k\neq k_r} a_{k} X^{k},
	\end{eqnarray*}
with $\lambda _{i}=|a_{j_i}|/(|a_{j_1}+\cdots +a_{j_{\ell }}|)$ for each $i=1,\dots ,\ell $.
We may then apply here the same proof as in Theorem \ref{SectorGenerala1}, with $k_r$ instead of $n$, since the polynomials  $X^{k_r} - \frac{|a_{j_i}|}{\lambda_i a_{k_r}} X^{j_i}$ above still have the form required in Lemma \ref{douamonoameGenerala}, as $k_r>j_i$ for each $r=1,\dots ,t$ and each $i=1,\dots ,\ell $.
\end{proof}
When the vertex $v$ given by (\ref{SectoralaFujiwaraGenerala}) is greater than $1$, one way to search for a smaller value is to use instead of the coefficients $a_{k_1},\dots ,a_{k_t}$ their sum, as in the following result.
\begin{theorem} \label{NumitorMare}
Let $f(X) = a_n X^n + \cdots + a_1 X + a_0$ be a polynomial with real coefficients, with $a_n > 0$, and denote by $j_1 < j_2 < \cdots < j_\ell$ the indices $j$ for which $a_j < 0$, and by $k_1<k_2<\cdots <k_t=n$ the indices $k>j_{\ell }$ for which $a_k>0$. Assume that $t\geq 2$ and let
	\begin{equation}\label{SectoralaFujiwaraGenerala2}
		v=\max \left \{ 1, 
\max \limits _{1\leq i\leq\ell}\left(\frac{|a_{j_1} + a_{j_2} + \cdots + a_{j_\ell}|}{a_{k_1}+a_{k_{2}}+\cdots +a_{k_t}}\right)^{\frac{1}{k_1-j_{i}}}\right \} .
	\end{equation}
Then $f$ has no zero in the sector $S_{v,\frac{\pi}{n}}$. 
\end{theorem}
\begin{proof}
Here we may write $f$ as
\begin{equation}\label{PrimaScriere}
f(X) =  \sum_{i=1}^{t}a_{k_i} X^{k_i} +\sum_{i = 1}^{\ell} a_{j_i} X^{j_i}+ \sum\limits_{a_{k} > 0,\ k<k_1} a_{k} X^{k}.
\end{equation}
Note that we may also write the first sum in (\ref{PrimaScriere}) as 
\[
(a_{k_1}+\cdots +a_{k_t})X^{k_1}+a_{k_2}(X^{k_2}-X^{k_1})+\cdots +a_{k_t}(X^{k_t}-X^{k_1}),
\]
so denoting $a_{k_1}+\cdots +a_{k_t}$ by $d$ we have
	\begin{eqnarray*}
		f(X) & = & dX^{k_1}+\sum_{i = 1}^{\ell} a_{j_i} X^{j_i}+ \sum_{i=2}^{t}a_{k_i}(X^{k_i}-X^{k_1})+\sum\limits_{a_{k} > 0,\ k<k_1} a_{k} X^{k}\\
 & = & \sum\limits_{i=1}^\ell \lambda_i d \left( X^{k_1} - \frac{|a_{j_i}|}{\lambda_i d} X^{j_i} \right) + \sum_{i=2}^{t}a_{k_i}(X^{k_i}-X^{k_1})+ \sum\limits_{a_{k} > 0,\ k<k_1} a_{k} X^{k},
	\end{eqnarray*}
with $\lambda _{i}=|a_{j_i}|/(|a_{j_1}+\cdots +a_{j_{\ell }}|)$ for each $i=1,\dots ,\ell $. Note that here we assigned no $\lambda _i$ to the terms $a_{k_i}(X^{k_i}-X^{k_1})$ that appear in the middle sum above, even if they are of the form in Lemma \ref{douamonoameGenerala}. However, we know from Lemma \ref{douamonoameGenerala} that each of these terms has positive real part in the sector $S_{1,\frac{\pi}{2n}}$, has positive imaginary part in the sector $S_{1,\frac{\pi}{n}}^+$, and negative imaginary part in the sector $S_{1,\frac{\pi}{n}}^-$. The proof then continues as in the case of Theorem \ref{SectorGenerala1}, with $d$ instead of $a_n$ and $k_1$ instead of $n$, the only difference being that when determining the right-most vertex of the sectors that appear, we have to also take into account the sectors $S_{1,\frac{\pi}{2n}}$, $S_{1,\frac{\pi}{n}}^+$, and $S_{1,\frac{\pi}{n}}^-$, that correspond to the polynomials $X^{k_i}-X^{k_1}$, $i=2,\dots ,t$.
\end{proof}
\begin{remark}\label{t=1}
We note that one may use (\ref{SectoralaFujiwaraGenerala2}) for the case that $t=1$ as well, but the presence of $1$ in the maximum in (\ref{SectoralaFujiwaraGenerala2}) gives in this case a result weaker than Theorem \ref{SectorGenerala1}. However, in order to simplify the analysis in Theorem \ref{VariatiiDeSemn} where we will take into account all the sign variations of $f$, we will use (\ref{SectoralaFujiwaraGenerala2}) for $t=1$ as well.
\end{remark}
\begin{corollary} \label{CoroInegalitatea1}
Let $f(X) = a_n X^n + \cdots + a_1 X + a_0$ be a polynomial with real coefficients, with $a_n > 0$, and denote by $j_1 < j_2 < \cdots < j_\ell$ the indices $j$ for which $a_j < 0$, and by $k_1<k_2<\cdots <k_t=n$ the indices $k>j_{\ell }$ for which $a_k>0$. If 
\begin{equation}\label{InegalitateaDintreSume}
a_{k_1}+a_{k_{2}}+\cdots +a_{k_t}\geq |a_{j_1} + a_{j_2} + \cdots + a_{j_\ell}|,
\end{equation}
then $f$ has no zero in the sector $S_{1,\frac{\pi}{n}}$. 
\end{corollary}
We will also prove a result that takes into account all the sign changes between consecutive (non-zero) coefficients of $f$. To state our result, we will first introduce some notations. 

For a polynomial $f(X) = a_n X^n + \cdots + a_1 X + a_0$ with real coefficients, $a_n>0$, we will consider the partition of its \textit{non-zero} coefficients into maximal ``connected" sequences containing only positive coefficients, or only negative coefficients, as illustrated below.  
\begin{equation}\label{FiguraCuSemne}
\underbrace{\stackrel{a_{P_1}}+\cdots\stackrel{a_{p_1}}+}_{S_1^+} \  \underbrace{\stackrel{a_{N_1}}-\cdots\stackrel{a_{n_1}}-}_{S_1^-} \  \underbrace{\stackrel{a_{P_2}}+\cdots\stackrel{a_{p_2}}+}_{S_2^+} \  \underbrace{\stackrel{a_{N_2}}-\cdots\stackrel{a_{n_2}}-}_{S_2^-} \;\; \cdots \;\; \underbrace{\pm\cdots\pm}_{S_k^\pm}
\end{equation}
Here $n=P_1\geq p_1>N_1\geq n_1>P_2\geq p_2>N_2\geq n_2>\cdots $ and $S_{j}^{+}=\sum\limits _{p_j\leq i\leq P_j}|a_i|$ while $S_{j}^{-}=\sum\limits _{n_j\leq i\leq N_j}|a_i|$ for each $j$.

For instance, for the polynomial $2X^9 + 5X^8 - 7X^5 - 3X^3 + X^2 - 8X - 1$ which has three sign changes, we have $k = 2$, $S_1^+ = 2 + 5 = 7$, $S_1^- = 7 + 3 = 10$, $S_2^+ = 1$, and $S_2^- = 8 + 1 = 9$.

If the number of sign changes between consecutive coefficients of $f$ is $s\geq 1$, say, then by (\ref{FiguraCuSemne}) we may write $f$ as $f(X)=g_1(X)+g_2(X)+\cdots +g_k(X)$ with $k=\lfloor \frac{s}{2}\rfloor +1$ and
\[
g_j(X)=\sum\limits _{p_j\leq i\leq P_j}a_iX^i+\sum\limits _{n_j\leq i\leq N_j}a_iX^i.
\]
The polynomials $g_j$ have a single sign change, except possibly for $g_k$, which might have no sign changes at all (i.e. in the case that $s$ is even). In this latter case, $g_k$ reduces to $g_k(X)=\sum\limits _{p_k\leq i\leq P_k}a_iX^i$, having only non-negative coefficients, and moreover, in this case it may even reduce to a constant, namely $a_0$. By Theorem \ref{NumitorMare} we see that a non-constant polynomial $g_j$ as above has no zeros in the sector $S_{v_j,\frac{\pi}{\deg g_j}}$ with
\begin{equation}\label{varfulvj}
v_j=\max \left\{1,\max\limits _{n_j\leq i\leq N_j}\left (\frac{S^{-}_j}{S^{+}_j}\right )^{\frac{1}{p_j-i}}\right\}.
\end{equation}

With these notations we have the following result.

\begin{theorem}\label{VariatiiDeSemn} A polynomial $f(X)= a_n X^n + \cdots + a_1 X + a_0$ of degree $n\geq 1$ with $a_n>0$ and having $s\geq 1$ sign changes between consecutive coefficients, cannot have zeros in the sector  $S_{v,\frac{\pi}{n}}$ with
$v=\max\{v_{1},\dots ,v_{\lfloor \frac{s}{2}\rfloor+1}\}$
and $v_j$ given by {\em (\ref{varfulvj})}.
\end{theorem}
\begin{proof}
We first write $f$ as a sum of polynomials $g_1,\dots ,g_k$ as before, with $k=\lfloor \frac{s}{2}\rfloor+1$ and the $g_i$'s having a single sign change (except possibly for $g_k$ that might have only non-negative coefficients, in case $s$ is even). We may then apply to each $g_i$ the method in the proof of Theorem \ref{NumitorMare} and finally consider the intersection of the zero-free sectors that we obtain. 
\end{proof}
In particular, one obtains the following corollary.
\begin{corollary}\label{CorolarVariatii}
Let $f(X)=a_n X^n + \cdots + a_1 X + a_0$ be a polynomial of degree $n\geq 1$ with real coefficients, $a_n>0$, having $s\geq 1$ sign changes between consecutive coefficients. Assume that $S^{+}_j>S^{-}_j$ for $j=1,\dots ,\lfloor \frac{s}{2}\rfloor+1$. Then $f$ has no zeros in the sector  $S_{1,\frac{\pi}{n}}$.
\end{corollary}

We will end this section with some considerations on how one can improve the estimates on the location of the roots of a polynomial $f$ by also looking at its reciprocal $\tilde{f}$. Recall that if $f(X)=a_0+a_1X+\cdots +a_nX^n$ is a polynomial with real coefficients of degree $n\geq 1$, $a_0\neq 0$, then $\tilde{f}(X)=X^{n}f(\frac{1}{X})=a_n+a_{n-1}X+\cdots +a_0X^n$. With this notation we have:
\begin{lemma}\label{reciprocul}
Let $f(X)$ be a polynomial with real coefficients of degree $n\geq 3$, $f(0)\neq 0$, and assume that its reciprocal $\tilde{f}$ has no roots in the sector $S_{\tilde{v},\frac{\pi}{n}}$ for some positive real number $\tilde{v}$. Then, $f$ has no zeros in the (open) lens-shaped region $L_{\tilde{v},\frac{\pi}{n}}$ given by
\begin{equation}\label{lentila}
x\in \left( 0,\frac{1}{\tilde{v}}\right),\quad |y|<\frac{-1}{2\tilde{v}\tan{\frac{\pi}{n}}}+\sqrt{\frac{1}{4\tilde{v}^2\sin ^2\frac{\pi}{n}}-\left(x-\frac{1}{2\tilde{v}}\right)^2}.
\end{equation}
\end{lemma} 
\begin{proof}
Recall that the roots of $\tilde{f}$ are precisely the inverses of the roots of $f$. One may check that
by inversion with respect to the origin of the complex plane ($z\mapsto \frac{1}{z}$) the sector $S_{\tilde{v},\frac{\pi}{n}}$ is mapped into the lens-shaped region $L_{\tilde{v},\frac{\pi}{n}}$ given by (\ref{lentila}) and vice-versa (see Figure \ref{fig2} below). The two open disks, whose intersection is precisely $L_{\tilde{v},\frac{\pi}{n}}$, have radius $\frac{1}{2\tilde{v}\sin \frac{\pi}{n}}$ and centers of coordinates $(\frac{1}{2\tilde{v}},\pm \frac{1}{2\tilde{v}}\cot \frac{\pi}{n})$.
Since none of the roots of $\tilde{f}$ belongs to the sector $S_{\tilde{v},\frac{\pi}{n}}$, $f$ cannot have roots in $L_{\tilde{v},\frac{\pi}{n}}$.
\end{proof}

\begin{figure}[h]
	\begin{tikzpicture}[>=stealth]
	\draw[->] (-0.5, -1.35) -- (-0.5, 1.35); \draw[->] (-1, 0) -- (7, 0);
	\draw[fill] (-0.5, 0) circle (0.05)  (2.5, 0) circle (0.05);
	\draw[thick, fill=gray, fill opacity = 0.2] (-0.5, 0) arc (120:60:3) arc(-60:-120:3);
	\draw (2.45, -0.08) node[anchor=135] {$\frac 1 {\tilde v}$};
	\draw[dashed] (-0.5, 0) -- (1.232, 1);
	\draw[dashed] (6, 0.8) -- (6.5, 0.975);
	\draw[dashed] (6, -0.8) -- (6.5, -0.975);

	\draw (0.25, 0.05) node[anchor=198] {\scriptsize${\pi}/{n}$};
	\draw (1.2, -0.9) node {$L_{\tilde v, \frac{\pi}{n}}$};
	\begin{scope}[shift={(3, 0)}]
	\draw[fill] (0.7, 0) circle (0.05);
	\draw[thick, fill=gray, fill opacity = 0.2] (3, 0.8) -- (0.7, 0) -- (3, -0.8);
	\draw (0.7, 0) node[anchor=45] {$\tilde{v}$};
	\draw (1.6, 0.05) node[anchor=200] {\scriptsize${\pi}/{n}$};
	\draw (1.8, -0.9) node {$S_{\tilde{v}, \frac{\pi}{n}}$};
	\end{scope}
	\end{tikzpicture}
	\caption{The lens-shaped region $L_{\tilde{v},\frac{\pi}{n}}$ and the sector $S_{\tilde{v}, \frac \pi n}$.}\label{fig2}
\end{figure}
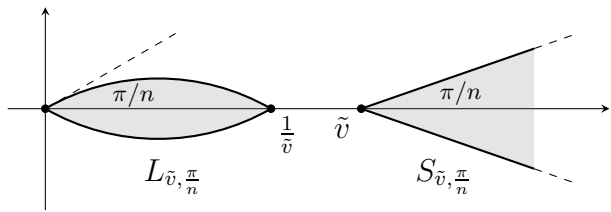
\vspace{-1em}



\begin{remark}\label{reuniunea}
As we shall see in the following section, searching for integers $m$ in zero-free regions of $f$ such that $f(m)$ is a prime may be a key factor when we test the irreducibility of $f$. Assume that $f$ is as in Lemma \ref{reciprocul}, and that we know by one of the results in this section that $f$ has no roots in a sector $S_{v,\frac{\pi}{n}}$ for some positive $v$. Then the information provided by Lemma \ref{reciprocul} may be of particular interest in case $\tilde{v}$ is sufficiently small, so that its inverse exceeds $v$. Thus, if $\frac{1}{\tilde{v}}>v$, one may prove that $S_{v,\frac{\pi}{n}}\cap L_{\tilde{v},\frac{\pi}{n}}\neq \emptyset $, and that the union $S_{v,\frac{\pi}{n}}\cup L_{\tilde{v},\frac{\pi}{n}}$ contains the sector $S_{0,\alpha}$ with angle $\alpha <\frac{\pi}{n}$ given by
\[
\alpha =\frac{\pi}{n}-\arctan \frac{\sin \frac{\pi}{n}}{\sqrt{\frac{1}{v\tilde{v}}-\sin^2\frac{\pi}{n}}},
\]
where $f$ has no roots. In particular, this shows that if $v\tilde{v}<1$, then $f$ has no positive real roots.
Furthermore, if $\tilde{v}<1$, the region $L_{\tilde{v},\frac{\pi}{n}}$ might contain points on the real axis with small integer coordinates $m$ such that $f(m)$ is a prime, which are not contained in the sector $S_{v, \frac \pi n}$. These points might be useful in testing the irreducibility of $f$, as we shall see later in Theorem \ref{DiskInLens}. 
\end{remark}

\section{Irreducibility criteria for polynomials with integer coefficients}\label{trei}
Many irreducibility criteria for polynomials $f$ with integer coefficients that use information on the prime factorization of $f(m)$ for some suitable integer argument $m$ often rely on upper bounds $B$ for the absolute values of the roots, usually expressed in terms of the absolute values of the coefficients of $f$. It is well known that if $f(m)$ is a prime number for some integer $m$ with $|m|>B+1$, then $f$ must be irreducible. One might naturally wonder if there exists a bound $B'$ depending only on the degree of $f$, that guarantees the irreducibility of $f$ once we know that $f$ assumes a prime value at an integer $m>B'$. This would be useful, for instance, in the case of polynomials with bounded degree, but whose roots have arbitrarily large absolute values. The aim of this section is to find families of polynomials for which such bounds $B'$ exist. This will be achieved by combining the results on zero-free regions obtained in Section \ref{MulteSectoare} with some ideas used in most of the irreducibility criteria mentioned in the Introduction. As we shall see, such bounds $B'$ exist for polynomials with non-negative integer coefficients, or for polynomials $f(X)=a_0+a_1X+\cdots +a_nX^n\in \mathbb{Z}[X]$ with all partial sums $a_n+a_{n-1}+\cdots +a_{n-j}$ non-negative, where $0\leq j\leq n$, for instance. In the general case of polynomials $f$ with arbitrary signs, we will content ourselves with bounds that alter the $B'$ obtained for polynomials with non-negative coefficients by an explicit term depending on the negative coefficients of $f$. The results that we will prove will actually consider the more general case that $f(m)=p^kq$ with $p$ prime, $k \ge 1$ and $q$ a small integer (compared to $m$). Our first result is the following:

\begin{theorem}\label{CircleInsideSector1}
Let $f(X)$ be a non-constant polynomial with integer coefficients, which has no roots in some sector $S_{v,\theta}$ with $0<\theta\leq \frac{\pi}{2}$. If $f(m)=pq$ with $p$ a prime number, $q$ and m integers, $q>0$ and $m>v+\frac{q}{\sin\theta }$, then $f$ is irreducible over $\mathbb{Q}$. Moreover, if $f$ has no rational roots, the same conclusion holds if $m>v+\frac{\sqrt{q}}{\sin\theta }$.
\end{theorem}
\begin{proof}
Assume to the contrary that $f$ factors as $f(X)=g(X)h(X)$ with $g(X),h(X)\in \mathbb{Z}[X]$ and $\deg g\geq 1, \deg h\geq 1$. Then, since $g(m)h(m)=pq$ and $p$ is a prime number, one of $g(m)$ and $h(m)$ must be divisible by $p$, say $p\mid h(m)$. As a consequence, we must have
\begin{equation}\label{modulmic}
|g(m)|\leq q. 
\end{equation}
Now let us assume that $f(X)=a_0+a_1X+\cdots +a_nX^n$ with $a_n\neq 0$, and that $f$ decomposes as $f(X)=a_n(X-\xi _1)\cdots (X-\xi _n)$, for some complex numbers $\xi _1,\dots ,\xi _n$.
Since $m>v+\frac{q}{\sin\theta }$, we observe that the closed disk centered at $(m,0)$ of radius $q$ lies inside the (open) sector $S_{v,\theta }$. As none of the roots $\xi _1,\dots ,\xi _n$ lies inside this sector, this shows that the distance between $(m,0)$ and each of the roots $\xi _1,\dots ,\xi _n$ must exceed $q$:
\[
|m-\xi _i|>q,\quad i=1,\dots ,n.
\]
On the other hand, if we assume without loss of generality that the polynomial $g$ factors as $g(X)=b_t(X-\xi _1)\cdots (X-\xi _t)$ with $t\geq 1$ and $b_t$ a divisor of $a_n$, we must have
\[
|g(m)|=|b_t|\cdot |m-\xi _1|\cdots |m-\xi _t|\geq |m-\xi _1|\cdots |m-\xi _t|>q^t\geq q,
\]
which contradicts (\ref{modulmic}). Thus, $f$ must be irreducible over $\mathbb{Q}$. 

If we already know that $f$ has no rational roots, then the degree of each factor of $f$ is at least $2$, in particular $t=\deg g\geq 2$, and on the other hand in this case our assumption that $m>v+\frac{\sqrt{q}}{\sin\theta }$ shows that 
the closed disk centered at $(m,0)$ of radius $\sqrt{q}$ lies inside the sector $S_{v,\theta }$. Therefore,
\[
|m-\xi _i|>\sqrt{q},\quad i=1,\dots ,n.
\]
We then deduce that
\[
|g(m)|=|b_t|\cdot |m-\xi _1|\cdots |m-\xi _t|>q^{\frac{t}{2}}\geq q,
\]
which contradicts (\ref{modulmic}) and completes the proof.
\end{proof}
By combining Theorem \ref{CircleInsideSector1} with the results on zero-free sectors in Section \ref{MulteSectoare}, one may obtain various irreducibility criteria, of which we will only state here as corollaries the simplest ones that correspond to the case that $q=1$. The first such result provides a lower bound for the positive integer $m$ that depends only on the degree of $f$.
\begin{corollary}\label{coro1}
Let $f(X)$ be a polynomial with non-negative integer coefficients, of degree $n\geq 2$. If $f(m)$ is a prime number for some integer $m>\frac{1}{\sin\frac{\pi}{n}}$, then $f$ is irreducible over $\mathbb{Q}$. 
\end{corollary}
\begin{proof}
Since $f$ has non-negative coefficients, by Lemma \ref{lemaImaginar(f)}
it cannot have any zero in the sector $S_{0,\frac{\pi }{n}}$, and the conclusion follows by Theorem \ref{CircleInsideSector1} with $q=1$, $v=0$ and $\theta =\frac{\pi}{n}$.
\end{proof}
\begin{remark}\label{remarca1}
In practice, to avoid working with denominators like $\sin\frac{\pi}{n}$, one may for instance use the fact that
\[
\frac{n}{\pi}+1>\frac{1}{\sin \frac{\pi}{n}}\quad {\rm for}\ n\geq 2,
\]
or to search for small rational constants $c>\frac{1}{\pi}$, such that $cn>\frac{1}{\sin\frac{\pi}{n}}$ under some mild restrictions on $n$. For the latter case, one may use the fact that $\frac{n}{3}\geq \frac{1}{\sin\frac{\pi}{n}}$ for $n\geq 6$, for instance. Thus, one may restate Corollary \ref{coro1} in the following slightly weaker, but more effective form:
\end{remark}
\begin{corollary}\label{coro2}
Let $f(X)$ be a polynomial with non-negative integer coefficients, of degree $n\geq 2$. If $f(m)$ is a prime number for some integer $m>\frac{n}{\pi}+1$, then $f$ is irreducible over $\mathbb{Q}$. 
\end{corollary}
By combining Lemma \ref{lemapozitivagenerala} with Corollary \ref{coro1}, one easily obtains irreducibility conditions for polynomials with integer coefficients whose partial sums $s_{f,j,1}$ are all non-negative:
\begin{corollary}\label{coro3}
Let $f(X)=a_0+a_1X+\cdots +a_nX^n\in \mathbb{Z}[X]$ be a polynomial of degree $n\geq 2$, with all partial sums $a_n+a_{n-1}+\cdots +a_{n-j}$ non-negative, where $0\leq j\leq n$. If $f(m)$ is a prime number for some integer $m>1+\frac{1}{\sin \frac{\pi}{n}}$, then $f$ is irreducible over $\mathbb{Q}$. 
\end{corollary}
\begin{proof} Let $g(X):=f(X+1)$. By Lemma \ref{lemapozitivagenerala} with $\alpha=1$, the polynomial $g(X)$ has non-negative coefficients, so if we prove that $g$ is irreducible, $f$ must be irreducible as well. We know that $f(m)$ is a prime number for some integer $m>1+\frac{1}{\sin \frac{\pi}{n}}$, so denoting $m-1$ by $m'$, we see that $g(m')$ is a prime number and $m'>\frac{1}{\sin \frac{\pi}{n}}$. By Corollary \ref{coro1}, $g$ must be irreducible, so $f$ too must be irreducible. This completes the proof.
\end{proof}
An immediate consequence of Theorem \ref{SectorGenerala1} and Theorem \ref{CircleInsideSector1} is given in the following result.
\begin{corollary} \label{coroFujiwaraCipi}
Let $f(X)=a_0+a_1X+\cdots +a_nX^n\in \mathbb{Z}[X]$ with $a_0\neq 0$, and denote by $j_1 < j_2 < \cdots < j_\ell$ all the indices $j$ for which $a_j < 0$. If $a_n>|a_{j_1}+\cdots + a_{j_\ell}|$, and $f(m)$ is a prime number for some integer $m > 1+\frac{1}{\sin{\frac{\pi}{n}}}$,
then $f$ is irreducible over $\mathbb{Q}$.
\end{corollary}
\begin{proof}
By Theorem \ref{SectorGenerala1}, $f$ will have no zeros in the sector $S_{1,\frac{\pi}{n}}$, and one applies then Theorem \ref{CircleInsideSector1} with $q=v=1$ and $\theta =\frac{\pi}{n}$.
\end{proof}
\begin{remark}\label{largemodulus}
We note here that if $a_0>\delta |a_1|+\cdots +\delta ^n|a_n|$ for a positive real $\delta $, then all the roots of $f$ must have absolute values exceeding $\delta $, so the larger $a_0$ is, the larger the maximum of the absolute values of the roots of $f$ is. On the other hand, in such case $a_0$ will not appear in $|a_{j_1}+ \cdots + a_{j_\ell}|$, so a lower bound like $1+\frac{1}{\sin{\frac{\pi}{n}}}$ in Corollary \ref{coroFujiwaraCipi} will hold for polynomials of degree $n$ with $a_n>|a_{j_1}+\cdots + a_{j_\ell}|$ and arbitrary large $a_0$, and hence with arbitrarily large absolute values of their roots. The lower bound $1+\frac{1}{\sin{\frac{\pi}{n}}}$ will therefore be much more efficient than the usual naive bound $1+\max\limits _{f(\xi )=0}|\xi |$.
\end{remark}

For polynomials with a single sign variation we will prove the following result.
\begin{corollary}\label{osinguravariatie}
Let $f(X)=a_0+a_1X+\cdots +a_nX^n\in \mathbb{Z}[X]$ be a polynomial of degree $n\geq 2$, with $a_n>0$, and having a single sign variation between consecutive coefficients. If $f(1)>0$ and $f(m)$ is a prime number for some integer $m>1+\frac{1}{\sin \frac{\pi}{n}}$, then $f$ is irreducible over $\mathbb{Q}$. 
\end{corollary}
\begin{proof}
Since $a_n>0$ and $f(1)>0$, condition (\ref{InegalitateaDintreSume}) in Corollary \ref{CoroInegalitatea1} is obviously satisfied, so $f$ will have no zeros in the sector $S_{1,\frac{\pi}{n}}$. Therefore, if $f(m)$ is a prime number for some integer $m>1+\frac{1}{\sin \frac{\pi}{n}}$, by Theorem \ref{CircleInsideSector1} with $v=q=1$ and $\theta =\frac{\pi}{n}$, $f$ must be irreducible over $\mathbb{Q}$. 
\end{proof}
We will consider now the case that $f(m)=p^kq$ with $p$ prime, $k$ some positive integer and $q$ some positive integer with $p\nmid q$. To test the irreducibility of $f$ in this case, it is often useful to also use some information on the derivative of $f$. In this respect, we will prove the following result.
\begin{theorem}\label{CircleInsideSector2}
Let $f(X)$ be a non-constant polynomial with integer coefficients, which has no roots in some sector $S_{v,\theta}$ with $0<\theta\leq \frac{\pi}{2}$. Assume that $f(m)=p^kq$ and $f'(m)=\pm p^{\ell}r$ with $p$ a prime number, $k$, $\ell$, $q$, $r$ and $m$ positive integers, $p\nmid qr$, and let $s:=\min (\ell, \frac{k}{2})$. If $m>v+\frac{p^{s}q}{\sin\theta }$, then $f$ is irreducible over $\mathbb{Q}$. Moreover, if $f$ has no rational roots, the same conclusion holds if $m>v+\frac{\sqrt{p^{s}q}}{\sin\theta }$.
\end{theorem}
\begin{proof}
Assume again, to the contrary, that $f$ may be written as $f(X)=g(X)h(X)$ with $g,h\in \mathbb{Z}[X]$ and $\deg g\geq 1, \deg h\geq 1$. Then, as 
$f'(X)=g'(X)h(X)+g(X)h'(X)$, we must simultaneously have $g(m)h(m)=p^kq$ and $g'(m)h(m)+g(m)h'(m)=\pm p^{\ell}r$. Let us now write $|g(m)|=p^a\alpha $ and $|h(m)|=p^b\beta $ with $a,b,\alpha ,\beta $ non-negative integers and $\alpha,  \beta $ not divisible by $p$. Then $\alpha \cdot \beta =q$, $a+b=k$, and without loss of generality we may assume that $a\leq b$, so $a\leq \frac{k}{2}$. On the other hand, we must also have $\textrm{min}(a,b)\leq \ell$, hence $a\leq \textrm{min}(\frac{k}{2},\ell)=s$. In particular, we conclude that 
\begin{equation}\label{majorant}
|g(m)|\leq p^{s}q.
\end{equation}
Assume as in the proof of Theorem \ref{CircleInsideSector1} that $f$ decomposes as $f(X)=a_n(X-\xi _1)\cdots (X-\xi _n)$, for some complex numbers $\xi _1,\dots ,\xi _n$. If $m>v+\frac{p^{s}q}{\sin\theta }$, the closed disk with center $(m,0)$ and radius $p^{s}q$ lies inside the sector $S_{v,\theta }$. Recalling that none of $\xi _1,\dots ,\xi _n$ lies in this sector, we see that the distance between $(m,0)$ and each of the roots $\xi _1,\dots ,\xi _n$ must exceed $p^{s}q$:
\[
|m-\xi _i|>p^{s}q,\quad \text{for}\ i=1,\dots ,n.
\]
On the other hand, if we assume again without loss of generality that the polynomial $g$ factors as $g(X)=b_t(X-\xi _1)\cdots (X-\xi _t)$ with $t\geq 1$ and $b_t$ a divisor of $a_n$, then we must have
\[
|g(m)|=|b_t|\cdot |m-\xi _1|\cdots |m-\xi _t|\geq |m-\xi _1|\cdots |m-\xi _t|>(p^{s}q)^t\geq p^{s}q,
\]
which contradicts (\ref{majorant}), so $f$ must be irreducible over $\mathbb{Q}$. 

Under the assumption that $f$ has no rational roots, the degree of each factor of $f$ must be at least $2$, so in particular $t = \deg g \geq 2$. On the other hand, in this second case our assumption that $m>v+\frac{\sqrt{p^{s}q}}{\sin\theta }$ shows that 
the closed disk $|z-m|\leq \sqrt{p^{s}q}$ lies inside the sector $S_{v,\theta}$. Thus,
\[
|m-\xi _i|>\sqrt{p^{s}q},\quad \text{for}\ i=1,\dots ,n.
\]
We then deduce that
\[
|g(m)|=|b_t|\cdot |m-\xi _1|\cdots |m-\xi _t|>(p^{s}q)^{\frac{t}{2}}\geq p^{s}q,
\]
which contradicts (\ref{majorant}) and completes the proof.
\end{proof}
In particular, for $q=1$ and $\ell=0$ one obtains the following result.
\begin{corollary}\label{coro4}
Let $f(X)$ be a polynomial with non-negative integer coefficients, of degree $n\geq 2$.  If for some integer $m>\frac{1}{\sin \frac{\pi}{n}}$ we have $f(m)=p^k$ with $p$ a prime number, $k$ a positive integer and $p\nmid f'(m)$, then $f$ is irreducible over $\mathbb{Q}$. 
\end{corollary}
\begin{proof}
Since $f$ has non-negative coefficients, by Lemma \ref{lemaImaginar(f)}
it has no roots in the sector $S_{0,\frac{\pi }{n}}$. The conclusion follows then by Theorem \ref{CircleInsideSector2} with $q=1$, $\ell=v=0$ and $\theta =\frac{\pi}{n}$.
\end{proof}
By combining now Lemma \ref{lemapozitivagenerala} with Corollary \ref{coro4}, we obtain other irreducibility conditions for polynomials with integer coefficients whose partial sums $s_{f,j,1}$ are all non-negative:

\begin{corollary}\label{coro5}
Let $f(X)=a_0+a_1X+\cdots +a_nX^n\in \mathbb{Z}[X]$ be a polynomial of degree $n\geq 2$, with all partial sums $a_n+a_{n-1}+\cdots +a_{n-j}$ non-negative, where $0\leq j\leq n$. If for some integer $m>1+\frac{1}{\sin \frac{\pi}{n}}$ we have $f(m)=p^k$ with $p$ a prime number, $k$ a positive integer and $p\nmid f'(m)$, then $f$ is irreducible over $\mathbb{Q}$. 
\end{corollary}
\begin{proof}  As before, if $g(X):=f(X+1)$, then by Lemma \ref{lemapozitivagenerala} with $\alpha=1$, $g$ must have non-negative coefficients, so if we prove that $g$ is irreducible over $\mathbb{Q}$, the same holds for $f$. Since $f(m)$ is a prime power $p^k$ for some integer $m>1+\frac{1}{\sin \frac{\pi}{n}}$, by letting $m'=m-1$, we see that $g(m')=p^k$ and $m'>\frac{1}{\sin \frac{\pi}{n}}$. On the other hand, $g'(m')=f'(m'+1)=f'(m)$, which is not divisible by $p$. Therefore, by Corollary \ref{coro4}, $g$ must be irreducible over $\mathbb{Q}$. The same must therefore hold for $f$, and the proof is complete.
\end{proof}

We will end this section with some irreducibility criteria that rely on Lemma \ref{reciprocul} and allow one to look at potentially smaller values of $m$ for which $f(m)$ might be a prime number, in the case that $f$ has no roots in a sector $S_{v,\frac{\pi}{n}}$ with large $v$, while its reciprocal has no roots in a sector $S_{\tilde{v},\frac{\pi}{n}}$ with sufficiently small $\tilde{v}$.
\begin{theorem}\label{DiskInLens}
Let $f(X)$ be a polynomial with integer coefficients of degree $n\geq 3$, with $f(0)\neq 0$, and assume that its reciprocal $\tilde{f}(X)$ has no roots in the sector $S_{\tilde{v},\frac{\pi}{n}}$ for some real number $\tilde{v}\in \left( 0,\frac{1}{2}\tan \frac{\pi}{2n}\right )$. If $f(m)$ is a prime number for some integer $m\in \left( \frac{1}{2\tilde{v}}-\delta, \frac{1}{2\tilde{v}}+\delta\right) $ with
\begin{equation}\label{delta}
\delta=\sqrt{1+\frac{1}{4\tilde{v}^2}-\frac{1}{\tilde{v}\sin \frac{\pi}{n}}},
\end{equation}
then $f$ is irreducible over $\mathbb{Q}$.
\end{theorem}
\begin{proof} 
First of all, one may check that $\delta $ is well defined, as $\tilde{v}\in \left( 0,\frac{1}{2}\tan \frac{\pi}{2n}\right )$. By Lemma \ref{reciprocul} we see that $f$ has no zeros in the lens-shaped region $L_{\tilde{v},\frac{\pi}{n}}$ given by (\ref{lentila}). We will search for conditions that allow a closed disk of radius $1$ centered at a point $(m,0)$ with $0<m<\frac{1}{\tilde{v}}$ to belong to the open region $L_{\tilde{v},\frac{\pi}{n}}$.  In this respect, the distance $d$ from the point $(m,0)$ to each of the two circular arcs defining the frontier of $L_{\tilde{v},\frac{\pi}{n}}$ is easily seen to be 
\[
d=\frac{1}{2\tilde{v}\sin \frac{\pi}{n}}-\sqrt{\left( m-\frac{1}{2\tilde{v}}\right )^2+\frac{1}{4\tilde{v}^2\tan ^2\frac{\pi}{n}}}.
\]
For $m=\frac{1}{2\tilde{v}}$ we have $d=
\frac{1}{2\tilde{v}}\tan \frac{\pi}{2n}$, which is larger than $1$, as $\tilde{v}\in \left( 0,\frac{1}{2}\tan \frac{\pi}{2n}\right )$. If $m\neq\frac{1}{2\tilde{v}}$, for $d$ to exceed $1$, we have to ask $m$ to satisfy $m^2\tilde{v}\sin\frac{\pi}{n}-m\sin\frac{\pi}{n}+1-\tilde{v}\sin\frac{\pi}{n}<0$, or equivalently, to
belong to the interval $\left( \frac{1}{2\tilde{v}}-\delta, \frac{1}{2\tilde{v}}+\delta\right) $, with $\delta$ given by (\ref{delta}), which is precisely the condition on the integer $m$ in the statement of the theorem. Thus the distance between $(m,0)$ and each of the roots $\xi _1,\dots ,\xi _n$ of $f$ is larger than $1$, as all the $\xi _i$'s lie outside $L_{\tilde{v},\frac{\pi}{n}}$. This may be written as $|m-\xi _i|>1$ for $i=1,\dots ,n$.

Assume now, as in the proof of Theorem \ref{CircleInsideSector1} that $f(X)=g(X)h(X)$ with $g(X),h(X)\in \mathbb{Z}[X]$ and $\deg g\geq 1, \deg h\geq 1$. Then, since $g(m)h(m)$ is a prime number, one of $g(m)$ and $h(m)$ must be of modulus $1$, say
\begin{equation}\label{modul1}
|g(m)|=1. 
\end{equation}
Let us denote the leading coefficient of $f$ by $a_n$ and assume that the polynomial $g$ factors as $g(X)=b_t(X-\xi _1)\cdots (X-\xi _t)$ with $t\geq 1$ and $b_t$ a divisor of $a_n$. Then we must have
\[
|g(m)|=|b_t|\cdot |m-\xi _1|\cdots |m-\xi _t|\geq |m-\xi _1|\cdots |m-\xi _t|>1,
\]
which contradicts (\ref{modul1}). Thus, $f$ must be irreducible over $\mathbb{Q}$. 
\end{proof}

\begin{corollary}\label{CoroDiskInLens}
Let $f(X)$ be a polynomial with integer coefficients of degree $n\geq 3$, with $f(0)\neq 0$, and assume that its reciprocal $\tilde{f}(X)$ has no roots in the sector $S_{\tilde{v},\frac{\pi}{n}}$ for some real number $\tilde{v}\in \left( 0,\frac{1}{2}\tan \frac{\pi}{2n}\right )$. If $f(m)$ is a prime number for some integer 
\begin{equation}\label{corectata}
m\in \left(\cot \frac{\pi}{2n}, \frac{1}{\tilde{v}}-\cot \frac{\pi}{2n}\right),
\end{equation}
then $f$ is irreducible over $\mathbb{Q}$.
\end{corollary}
\begin{proof} 
One may check that $\tilde{v}\in \left( 0,\frac{1}{2}\tan \frac{\pi}{2n}\right )$ implies that
$\delta $ in Theorem \ref{DiskInLens} satisfies $\delta > \frac{1}{2\tilde{v}}-\cot \frac{\pi}{2n}$, so we have the inclusion of intervals
\[
\left(\cot \frac{\pi}{2n}, \frac{1}{\tilde{v}}-\cot \frac{\pi}{2n}\right)\subset \left(\frac{1}{2\tilde{v}}-\delta, \frac{1}{2\tilde{v}}+\delta \right).
\]
The conclusion follows now from Theorem \ref{DiskInLens}.
\end{proof}
To obtain a more effective, but slightly weaker result that uses no trigonometric functions, one may use the fact that $\cot \frac{\pi}{2n}<\frac{2n}{\pi}$ for $n\geq 1$, as follows.
\begin{corollary}\label{CoroDiskInLensEffective}
Let $f(X)$ be a polynomial with integer coefficients of degree $n\geq 3$, with $f(0)\neq 0$, and assume that its reciprocal $\tilde{f}(X)$ has no roots in the sector $S_{\tilde{v},\frac{\pi}{n}}$ for some real number $\tilde{v}\in \left( 0,\frac{\pi}{4n}\right )$. If $f(m)$ is a prime number for some integer 
\begin{equation}\label{corectataEffective}
m\in \left(\frac{2n}{\pi}, \frac{1}{\tilde{v}}-\frac{2n}{\pi}\right),
\end{equation}
then $f$ is irreducible over $\mathbb{Q}$.
\end{corollary}

We note that there exist polynomials $f$ of degree $n$ with arbitrarily large absolute values of their roots, for which their reciprocals $\tilde{f}$ have no roots in sectors $S_{\tilde{v},\frac{\pi}{n}}$ with reasonably small length of the interval in (\ref{corectata}). For such polynomials it might be therefore more efficient to use Corollary \ref{CoroDiskInLens} instead of searching for integers $m$ with $f(m)$ prime and $m$ exceeding the naive bound $1+\max\limits _{f(\xi)=0}|\xi|$. We will give examples of such polynomials in the following section.

By combining Theorem \ref{CircleInsideSector1} with $q=1$ and Corollary \ref{CoroDiskInLens} we finally obtain: 
\begin{corollary}\label{CoroCombinat}
Let $f(X)$ be a polynomial with integer coefficients of degree $n\geq 3$, with $f(0)\neq 0$, and assume that $f$ and its reciprocal $\tilde{f}$ have no roots in the sectors $S_{v,\frac{\pi}{n}}$ and $S_{\tilde{v},\frac{\pi}{n}}$, respectively, for some positive real numbers $v$ and $\tilde{v}$, with $\tilde{v}<\frac{1}{2}\tan \frac{\pi}{2n}$. If $f(m)$ is a prime number for some integer
\begin{equation}\label{vsivtilde}
m\in \left(\cot \frac{\pi}{2n}, \frac{1}{\tilde{v}}-\cot \frac{\pi}{2n}\right)\cup \left(v+\frac{1}{\sin \frac{\pi}{n}},\infty\right),
\end{equation}
then $f$ is irreducible over $\mathbb{Q}$.
\end{corollary}
\begin{remark}\label{SiMaiFina}
If $\cot \frac{\pi}{n}<v<\frac{1}{\tilde{v}}-\cot \frac{\pi}{2n}-\frac{1}{\sin\frac{\pi}{n}}$, condition 
(\ref{vsivtilde}) reduces to $m>\cot \frac{\pi}{2n}$, while if $v\leq\cot \frac{\pi}{n}$, condition (\ref{vsivtilde}) reduces to $m>v+\frac{1}{\sin\frac{\pi}{n}}$, and in this case there is no need to impose the condition $\tilde{v}<\frac{1}{2}\tan \frac{\pi}{2n}$.
\end{remark}
The reader may of course prove more general results of this kind, by allowing $f(m)$ to be of the form $pq$, or $p^kq$ with $k\geq 2$, as in Theorem \ref{CircleInsideSector1} and Theorem \ref{CircleInsideSector2}. 

\section{Examples}\label{examples}
1) Let $f(X)$ be a polynomial with non-negative integer coefficients of degree $n\geq 2$. Then for any integer $m>\frac{1}{\sin \frac{\pi}{n}}$ and any prime number $p\geq f(m)-f(0)$, the polynomials 
\[
g_{m,p}(X)=f(X)+p-f(m)
\]
are irreducible over $\mathbb{Q}$. To prove this, we see that $g_{m,p}(X)$ is of degree $n$, has non-negative coefficients, and $g_{m,p}(m)=p$, a prime number, and one can apply Corollary \ref{coro1}.

2) Let $f(X)=a_0+a_1X+\cdots +a_nX^n\in \mathbb{Z}[X]$ be a polynomial of degree $n\geq 2$, with all partial sums $a_n+a_{n-1}+\cdots +a_{n-j}$ non-negative, where $0\leq j\leq n$. Then for any integer $m>1+\frac{1}{\sin \frac{\pi}{n}}$ and any prime number $p\geq f(m)-f(1)$, the polynomials 
\[
g_{m,p}(X)=f(X)+p-f(m)
\]
are irreducible over $\mathbb{Q}$. Using the notations in the proof of Lemma \ref{lemapozitivagenerala}, we see that $g_{m,p}(X)$ is of degree $n$, and its partial sums $s_{g_{m,p},j,1}$ coincide with $s_{f,j,1}$ for $j=0,\dots ,n-1$, while 
\[
s_{g_{m,p},n,1}=s_{f,n,1}+p-f(m)=a_n+a_{n-1}+\cdots +a_{0}+p-f(m)=p-f(m)+f(1),
\]
which is also non-negative. Since $g_{m,p}(m)=p$, a prime number, one can apply Corollary \ref{coro3}.

3) Let $f(X)$ be a polynomial with non-negative integer coefficients of degree $n\geq 2$. Then for any integer $m>\frac{1}{\sin \frac{\pi}{n}}$, any integer $k\geq 2$ and any prime number $p>f'(m)$, the polynomials 
\[
g_{m,p}(X)=f(X)+p^{k}-f(m)
\]
are irreducible over $\mathbb{Q}$. We first prove that $g_{m,p}(X)$ has non-negative coefficients. As $f$ has this property, it remains to check that $g_{m,p}(0)\geq 0$, or equivalently, that $p^{k}\geq f(m)-f(0)$. Since $k\geq 2$ and $p>f'(m)$, it suffices to prove that $f'(m)^{2}\geq f(m)-f(0)$. So let us write $f(X)=a_0+a_1X+\cdots +a_nX^n$. Using the fact that the $a_i$'s are non-negative, we deduce that
\begin{eqnarray*}\label{f'mare}
f'(m)^2 & =& (na_nm^{n-1}+(n-1)a_{n-1}m^{n-2}+\cdots +2a_2m+a_1)^2 \\ 
& \geq & na_nm^{n-1}(na_nm^{n-1}+(n-1)a_{n-1}m^{n-2}+\cdots +2a_2m+a_1)\\
 & \geq & a_nm^n+a_{n-1}m^{n-1}+\cdots +a_1m=f(m)-f(0),
\end{eqnarray*}
with the last inequality holding since $na_nm^{n-1}\cdot ia_im^{i-1}\geq a_im^i$ for each $i=1,\dots ,n$. 

We finally observe that $g'_{m,p}(X)=f'(X)$, so $p\nmid g'_{m,p}(m)$, and since $g_{m,p}(m)=p^{k}$, one can apply Corollary \ref{coro4}.

4) If we can write a prime number as $81-27a+b$ for some positive integers $a,b$ with $b>216a$, then the polynomial $f(X)=X^4-aX^3+b$ is irreducible over $\mathbb{Q}$. To see this, first observe that by Theorem \ref{SectorGenerala1}, its reciprocal $\tilde{f}(X)=bX^4-aX+1$ has no roots in the sector $S_{\tilde{v},\frac{\pi}{4}}$ with $\tilde {v}=(\frac{a}{b})^{\frac{1}{3}}$. As $b>216a$, we see that $\tilde{v}<\frac{1}{6}<\frac{\sqrt{2}-1}{2}=\frac{1}{2}\tan \frac{\pi}{8}$, so by Corollary \ref{CoroDiskInLens}, if $f(m)$ is a prime number for an integer $m\in (1+\sqrt{2}, (\frac{b}{a})^{\frac{1}{3}}-1-\sqrt{2})$, then $f$ must be irreducible over $\mathbb{Q}$. A suitable such candidate for $m$ is obviously $3$, as it belongs to this interval. Note that by Theorem \ref{SectorGenerala1}, $f$ has no roots in the sector $S_{a,\frac{\pi}{4}}$, so if $f(m)$ is a prime number for some integer $m>a+\frac{1}{\sin \frac{\pi}{4}}$ one may also conclude that $f$ is irreducible, by using Theorem \ref{CircleInsideSector1}. However, for large $a$, and hence large $b$ too, testing this might be more difficult than testing that $f(3)$ is a prime number. For an explicit example, one may for instance consider the polynomial $f(X)=X^4-10X^3+2162$, which is irreducible, as $f(3)=1973$ is a prime number.


\begin{thebibliography}{99}  
\bibitem{Ballieu} R. Ballieu, \textit{Sur les limitations des racines d'une 
\'equation alg\'ebrique}, Acad. Roy. Belg. Bull. Cl. Sci. (5) 33 (1947), 747--750.

\bibitem{BMS} P. Batra, M. Mignotte, and D. \c Stef\u anescu, {\it Improvements 
of Lagrange's bound for polynomial roots}, J. Symbolic Comput. 82 (2017), 19--25.

\bibitem{BDN3} A. Bodin, P. D\`ebes and S. Najib, {\it Prime and coprime values of polynomials.} Enseign. Math. (2) 66 (2020), 169--182.
 
\bibitem{Bonciocat1}  A.I. Bonciocat and N.C. Bonciocat, \textit{The irreducibility of polynomials that have one large coefficient and take a prime value}, 
Canad. Math. Bull. 52 (2009), no. 4, 511--520.

\bibitem{Bonciocat2} A.I. Bonciocat, N.C. Bonciocat, and A. Zaharescu, \textit{On the irreducibility of polynomials that take a prime power value}, 
Bull. Math. Soc. Sci. Math. Roumanie 54 (102) (2011), no. 1, 41--54.


\bibitem{Brillhart}  J. Brillhart, M. Filaseta, and A. Odlyzko, \textit{On an irreducibility theorem of A. Cohn}, Canad. J. Math. 33 (1981), no. 5, 1055--1059.

\bibitem{CDF} M. Cole, S. Dunn,  and M. Filaseta, 
{\it Further irreducibility criteria for polynomials with non-negative coefficients}, Acta Arith. 175 (2016), no. 2, 137--181. 

\bibitem{Cowling1} V.F. Cowling and  W.J. Thron, \textit{Zero-free regions of polynomials}, 
Amer. Math. Monthly 61 (1954), 682--687.

\bibitem{Cowling2} V.F. Cowling and  W.J. Thron, \textit{Zero-free regions of polynomials}, 
J. Indian Math. Soc. (N.S.) 20 (1956), 307--310.

\bibitem{Dorwart}  H.L. Dorwart {\it Irreducibility of polynomials}, Amer. Math. Monthly 42 (1935), no. 6, 369--381.

\bibitem{Filaseta1}  M. Filaseta, \textit{A further generalization of an irreducibility theorem of A. Cohn}, Canad. J. Math. 34 (1982), no. 6, 1390--1395.

\bibitem{Filaseta2}  M. Filaseta, \textit{Irreducibility criteria for polynomials with non-negative coefficients}, Canad. J. Math. 40 (1988), no. 2, 339--351.

\bibitem{Fujiwara}  M. Fujiwara, \textit{\"{U}ber die obere Schranke des absoluten Betrages der Wurzeln einer algebraischen Gleichung}, T\^{o}hoku
Math. J. 10 (1916), 167--171.

\bibitem{Girstmair}  K. Girstmair, \textit{On an Irreducibility Criterion of M. Ram Murty}, Amer. Math. Monthly 112 (2005), no. 3, 269--270.

\bibitem{Guersenzvaig} N.H. Guersenzvaig, \textit{Simple arithmetical criteria for irreducibility of polynomials with integer coefficients}, Integers 13 (2013), 1--21.

\bibitem{GS} N.H. Guersenzvaig and F. Szechtman, \textit{Roots multiplicity and square free factorization of polynomials using companion matrices}, Linear Algebra Appl. 436 (9) (2012), 3160-3164.

\bibitem{Kempner} A.J. Kempner, \textit{Ueber die Separation komplexer Wurzeln algebraischer Gleichungen}, Math. Ann. 85 (1992), 49--59.

\bibitem{Kioustelidis} J.B. Kioustelidis, \textit{Bounds for positive roots of polynomials}, J. Comput. Appl. Math. 16 (1986), 241--244.

\bibitem{Kojima} T. Kojima, \textit{On a theorem of Hadamard's and its application}, T\^{o}hoku Math. J. 5 (1914), 54--60.

\bibitem{Laszlo} L. L\' aszl\' o, \textit{Imaginary Part Bounds on Polynomial Zeros}, Linear Algebra Appl. 44 (1982), 173--180.

\bibitem{Marden} M. Marden, \textit{Geometry of polynomials}, Mathematical Surveys and Monographs No. 3, American Mathematical Society, Providence, RI, 1966.


\bibitem{MS} M. Mignotte and D. \c Stef\u anescu, \textit{Polynomials. An algorithmic approach}, Springer 1999.

\bibitem{Ore} O. Ore, {\it Einige Bemerkungen 
\"uber Irreduzibilit\"at}, Jahresbericht der Deutschen Mathematiker-Vereinigung 44 (1934), 147--151.

\bibitem{Perron} O. Perron, \textit{Algebra. II Theorie der algebraischen Gleichungen}, Walter de Gruyter \& Co., Berlin, 1951.

\bibitem{PolyaSzego}  G. P\'{o}lya and  G. Szeg\"{o}, \textit{Aufgaben und Lehrs\"{a}tze aus der Analysis}, Springer-Verlag, Berlin, 1964.

\bibitem{RamMurty}  M. Ram Murty, \textit{Prime numbers and irreducible polynomials}, Amer. Math. Monthly 109 (2002), no. 5, 452--458.

\bibitem{Stackel} P. St\"ackel, {\it Arithmetischen Eigenschaften ganzer Funktionen}, Journal f\"ur Mathematik 148 (1918), 101--112.

\bibitem{DStef1} D. \c Stef\u anescu, {\it On bounds for real roots of polynomials}, Rom. Journ. Phys. 58 (2013), nos. 9--10, 1428--1435.

\bibitem{Turan} P. Tur\' an, \textit{Hermite-expansion and strips for zeros of polynomials}, Arch. Math. (Basel) 5 (1954), 148--152.

\bibitem{Weisner} L. Weisner, {\it Criteria for the irreducibility of polynomials}, Bull. Amer. Math. Soc. 40 (1934), 864--870.

\end{thebibliography}
\end{document}